\documentclass[12pt,openbib]{article}
\usepackage{amsmath}
\usepackage[utf8]{inputenc}
\usepackage[english]{babel}
\usepackage{graphicx}
\usepackage{amsmath,amsthm,amsfonts,amssymb,amscd}
\usepackage{float}
\usepackage{caption}
\usepackage{amsfonts}
\usepackage{amsfonts,amssymb}
\usepackage{amssymb}
\usepackage{latexsym}
\usepackage{euscript}
\usepackage{enumerate}
\usepackage{graphics}

\def\Z{{\mathbb Z}}
\def\0{{\mathbf 0}}
\def\1{{\mathbf 1}}

\newtheorem{thm}{Theorem}

\newtheorem{rk}{Remark}

\begin{document}

\author{Vassily Olegovich Manturov\footnote{Moscow Institute of Physics and Technology,
Kazan Federal University, Northeastern University (China), {\bf {vomanturov @ yandex.ru}}}}

\title{A free-group valued invariant of free knots}

\maketitle

\begin{abstract}
{The aim of the present paper is to construct series of invariants
of free knots (flat virtual knots, virtual knots) valued in
free groups (and also free products of cyclic groups)

}
\end{abstract}



\section{Introduction. Preliminaries}

Knots are encoded by {\em Gauss diagrams} modulo {\em Reidemeister moves}. For {\em classical knots} one imposes the
{\em planarity} restriction on the set of Gauss diagrams. Once we forget planarity, we get {\em virtual knots}. Gauss diagrams have
all chords endowed with two bits of information: {\em local writhe} number and the direction from one point (``over'') to the other
point (``under''). If we abandon these two bits of information, we get arbitrary Gauss diagrams modulo the three Reidemeister moves
as depicted in Fig. 1, top.

\begin{figure}
\centering\includegraphics[width=220pt]{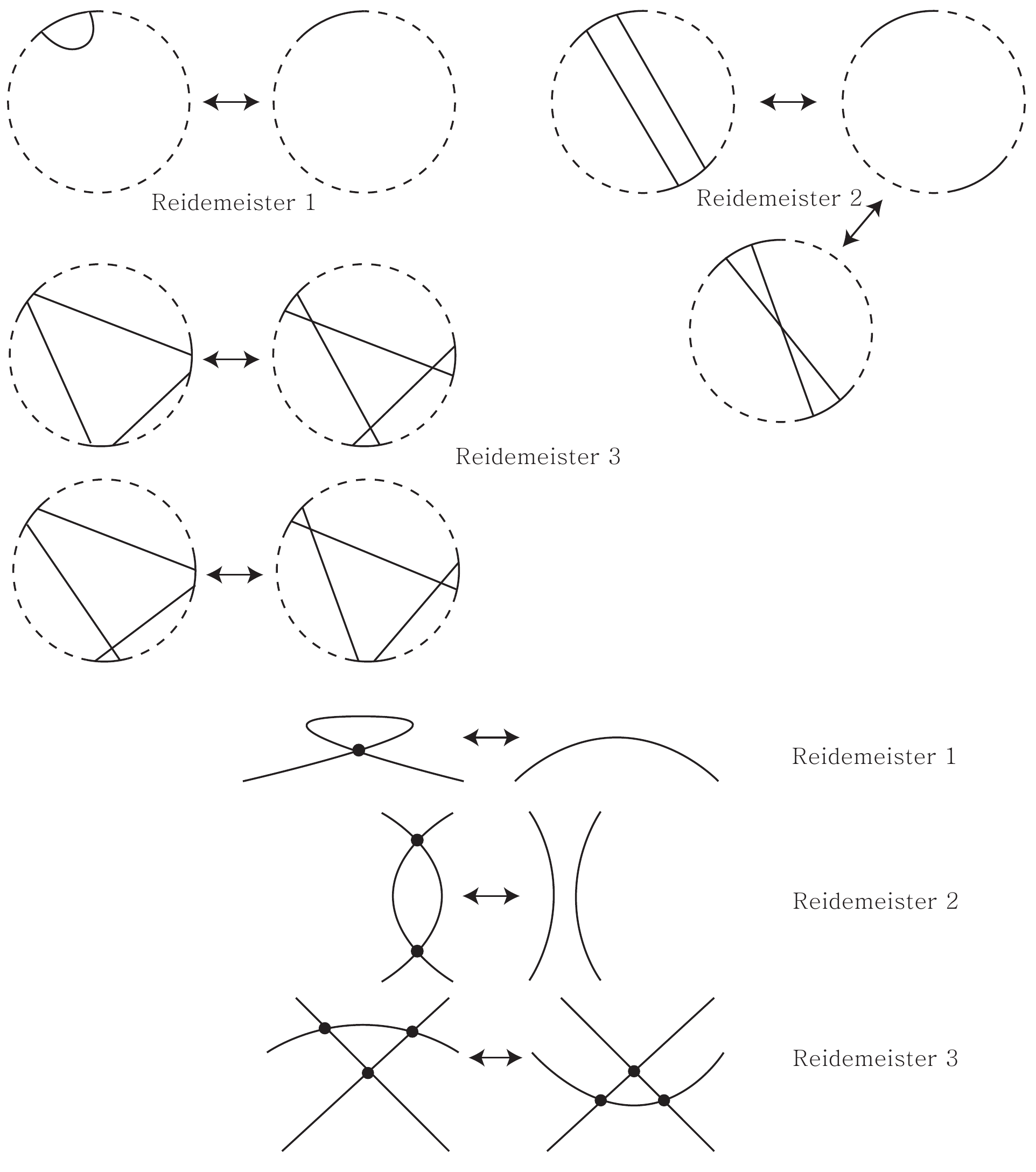}
\caption{Gauss diagrams and Reidemeister moves}
\label{Fig1}
\end{figure}

The other way to encode Gauss diagrams is to use $1$-component framed $4$-graphs. For more details, see \cite{Parity}.
In the bottom part of Fig. \ref{Fig1} we show the Reidemeister moves in the language of framed $4$-graphs.

One can also define {\em long free knots} by breaking a framed $4$-graph at a point and pulling the two ends to the infinity;
certainly, for long knots it is not allowed to perform moves {\em over the infinity}.

We shall often switch from free knots to long free knots; the main advantage of free knots is the presence of a {\em reference point};
one can naturally see that if we associate a {\em word} to a free knot, we can associate a {\em cyclic word} to a (compact free knot);
yet another advantage of free knots is the possibilities to fix some ``coordinate system'' which allows to ``justify'' the names of letters;
actually, for compact free knots the word associated to letters will be sometimes well defined only up to some remaining
letters, which is some outer automorphism of the group in question; for long free knots there is no such problem.

Here by {\em components} of a framed $4$-graph we mean {\em unicursal curves}.
The same can be used to define {\em free links}: we take similar (Gauss) chord diagrams with many core circles
and arbitrary framed $4$-graphs modulo Reidemeister moves.

Free knots (initially introduced by Turaev \cite{Turaev}) under the name of {\em homotopy classes of Gauss words}
were first conjectured to be trivial; it was disproved in \cite{Parity}, see also \cite{Gibson}.

The {\em parity theory} turns out to be a very powerful tool in the study of free knots and links, and in
\cite{Parity} (and later in \cite{IMN}, \cite{InvariantsAndPictures}) we demonstrated the following features of parity.

\begin{enumerate}

\item One can construct picture-valued invariants of {\em free} knots
(by {\em pictures} we mean linear combinations of free knot diagrams)

\item There is an invariant called {\em parity bracket}, $[\cdot]$, which for some free knot diagrams
$K$ gives us

$$[K]=K.$$

\item The above formula gives rise to the following principle: {\em if a diagram is complicated enough
then it realises itself}. The simplest variant of this principle is: If all crossings of the diagram $K$ are odd
and there is no way to apply the second decreasing move to $K$ then any diagram $K'$ equivalent to $K$
admits a smoothing which is identical to $K$.

\item We often see that {\em locally minimal} free knot diagrams (those which do not
admit decreasing Reidemeister moves) are {\em globally minimal} in a strong sense (e.g., are contained in any equivalent diagrams).

\item Such results can be partially extended to the realm of cobordisms: {\em if free knots can be capped by some ``folded'' disc
then they can be capped by a disc in a simple way}; for more details see \cite{ManturovFedoseev}.

\item Knots with parity behave like {\em links}: one can explicitly construct $2$-fold coverings over free/flat/virtual/pseudo/quasi
knots by links. \cite{IMN}.

\end{enumerate}

Actually, the parity appears whenever one has some additional ``$\Z_{2}$-homology'' condition: roughly speaking,
one has ``even'' and ``odd'' homology classes associated to crossings in a certain way.

The property ``if a diagram is complicated enough then it realises itself'' makes {\em free knots similar to  free groups}
(see \cite{InvariantsAndPictures}). Indeed, if we have a free group $F^{n}$ in generators, say, $a_{1},\cdots, a_{n}$,
then for any {\em reduced} word $W$ any word $W'$ equivalent to $W$ contains $W$ as a ``subword'', i.e., $W$ can be obtained
from $W'$ by iterative cancellation of some pairs of $a_{j}a_{j}^{-1}$ or $a_{j}^{-1}a_{j}$.

Hence, there is an obvious analogy between {\em free knots} and {\em free words}; here we have:

\begin{enumerate}

\item Reductions $\Longleftrightarrow$ second decreasing Reidemeister moves;

\item Reduced words $\Longleftrightarrow$ minimal representatives of diagrams.

\end{enumerate}

Certainly, the above correspondence is not $1\longleftrightarrow1$. We want to exploit this correspondence in both directions: we try
to study groups by using knots and knots by using groups.

Here we concentrate on the

\begin{enumerate}

\item Free knots $\to$ elements of a free group;

\item (Compact) free knots $\to$ conjugacy classes of a free group.

\end{enumerate}

The general plan to be realised for topological problems is as follows:
\begin{enumerate}

\item First, we look at some ``homological'' conditions;

\item These conditions allow one to get a certain ``labels'' $l_{1},l_{2},\cdots, l_{n}$;

\item Construct the free group $\mathbb{Z}_{l_1} * \cdots * \mathbb{Z}_{l_n}$ or the free product $\mathbb{Z}_{2}*\cdots *\mathbb{Z}_{n}$
of generators $a_{1},\cdots, a_{n}$ corresponding to $l_{1}\cdots l_{n}$.

\item Try to prove that when we apply equivalence (say, Reidemeister moves) to the initial topological
objects (say, knots) then different generators in the image {\em do not commute}.

\end{enumerate}

In fact, the last step is the most complicated one. Actually, it is not too difficult to grasp some
homological information for crossings so that the two crossings related by the second Reidemeister move
are cancelled as follows:

$$ a_{i} a_{i}^{-1}= 1.$$

Having some non-trivial homological information, one can avoid writing down generators corresponding to the first
Reidemeister moves: if there were such generators, say, $a_{j}$, they would give rise to the painful \footnote{Actually,
the relation $a_{j}^{2}$ is less painful, but sometimes we may count say, only left ends and not right chord ends} relation

$$a_{j}=1.$$

The hardest problem here is the existence of the third Reidemeister move. It gives rise to the some transformation of
a word in {\em three places}.

\begin{figure}
\centering\includegraphics[width=250pt]{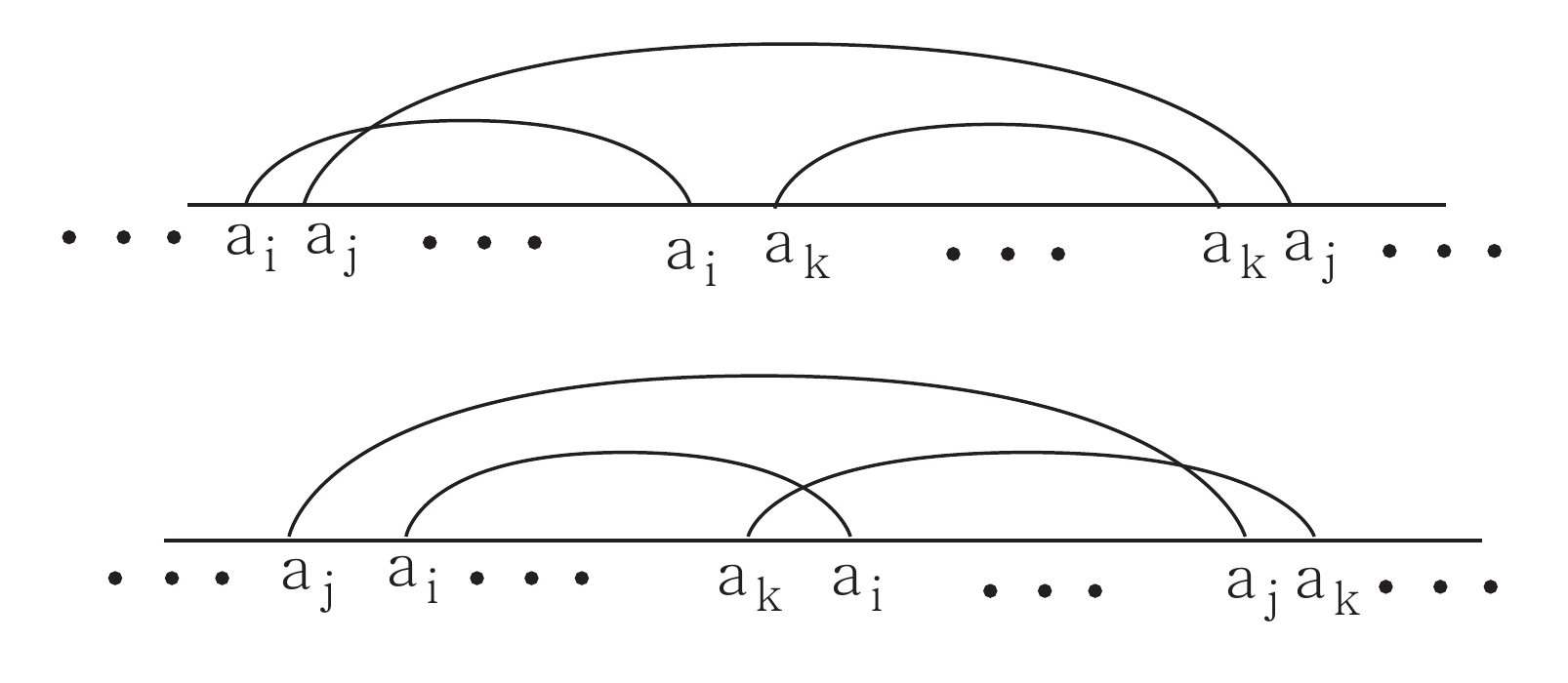}
\caption{Transformation of a word in three places}
\label{ThreePlaces}
\end{figure}

In each of these three places the word is transformed by the commutativity relation
$a_{j}a_{i}=a_{i}a_{j}$ which does not allow one to construct a genuine {\em free group invariant}.

In the present paper, we shall show the very first instance how this difficulty can be overcome
and how one can get

\begin{enumerate}

\item From homology to homotopy;

\item From free abelian groups to free groups.

\end{enumerate}

{\bf In the present paper, we shall construct various types of letters which are all due to the existence of just
one source of homological information, the Gaussian parity, which exists for free knots.}

This paper is one of the series of papers having the goal to study classical objects (knots, links,
their cobordisms) by using virtual knot methods. This is in some sense ``the last piece of the whole puzzle'':
we get invariants of free knots. The first pieces of the puzzle will start from classical knots and transform
them to ``virtual'' objects having non-trivial parity etc.

Here we note that knots in the cylinder already possess enough structure to use virtual knot theory methods;
on the other hand, cylinders naturally appear in the study of classical knots and links (For more details,
see \cite{InvariantsAndPictures}).

\subsection{The structure of the paper}

The paper is organised as follows. In Section 2, we construct (some old) the invariants of free knots valued in free groups and prove its
invariance.

Also, we informally discuss various ways of generalisation for this invariant.

In Section 3, we calculate the main example which shows that free knots can give rise to ``highly non-commutative''
words in the free group.

We conclude the paper by a discussion of further directions like cobordisms and Legendrian knots.

\subsection{Acknowledgements}

The author is grateful to Kim Seongjeong, Kim Sera and L.Kauffman and J.Wu for various fruitful discussions.

This work  was funded by the development program of the Regional Scientific and Educational Mathematical
Center of the Volga Federal District, agreement N 075-02-2020.

\section{The construction of the invariants. \\ The invariance proof}

In the present section, we describe the whole procedure how to construct the invariant of free knots valued in free groups.
To make the process clear and accessible for other types of parities (say, homological), we describe it step-by step.

\begin{enumerate}

\item The first step will consist in construction of some {\em integer-valued} invariant which can be reduced
to some {\em count}: we count the number of chords which are ``non-trivial'' by some reasons.

\item The second step will give rise to the invariant valued in the infinite
dihedral \footnote{there is a natural bijection between the elements of the group below and the elements of the infinite dihedral group} group which we write down as: $G=\langle a,b,b'| a^{2}=b^{2}=(b')^{2}=1, ab=b'a\rangle.$

\item The third step will give rise to an invariant valued in a group of exponential growth.
The key ingredient is the possibility to get rid of the relations $b^{2}=b'^{-2}$ by {\em renaming letters
and rewriting words.}
\footnote{Actually, this procedure can be accurately defined by using wreath products of groups with $\Z_{2}$ but
we don't want to give much detail in this short note.}


\end{enumerate}

\subsection{The first step}

Let $K$ be a Gauss diagram; we say that two chords $c,c'$ are {\em linked} if the endpoints of $c$
alternate with the endpoints of $c'$. Sometimes we shall write $\langle c,d\rangle = 1\in \Z_{2}$ if
$c,d$ are linked and $\langle c,d\rangle = 0 \in \Z_{2}$, otherwise.

A chord $c$ is {\em even} if the number of chords it is linked with, is {\em even}, and {\em odd}
if it is odd.

Now, the reader can immediately see that:
\begin{enumerate}
\item the chord does not change its parity when a Reidemeister-3 move is applied;
\item if two chords are cancelled by a Reidemeister-2 move then they have the same parity;
\item A chord taking part in the first Reidemeister move is even.
\item The number of odd chords of a Gauss diagram is even.
\end{enumerate}





Let $K$ be an oriented long knot diagram. We enumerate the endpoints as they appear according to the orientation;
we say that an odd chord is {\em of the first type} if it is linked with evenly many even chords
and {\em of the second type} otherwise.

Moreover, we say that an odd chord is {\em of the first sort} if both points are
in odd positions and it is linked with evenly many
even chords or if both endpoints of it are even and it is linked with oddly many even chords.

Otherwise we say that the odd chord is {\em of the second sort}.


Actually, one can make a distinction between various even chords by introducing the notion of {\em hierarchy} where odd chords are
of hierarchy $0$, even chords of the first type are of hierarchy $2$ and other even chords of type $2$ have
hierarchy $n$, where $n\in \{2,3,\cdots \}\cup \{\infty\}$.
\cite{Hierarchy}.

The results of the present paper can be definitely extended if we use more information about the hierarchy, but
here the type will be sufficient.

\begin{figure}
\centering\includegraphics[width=250pt]{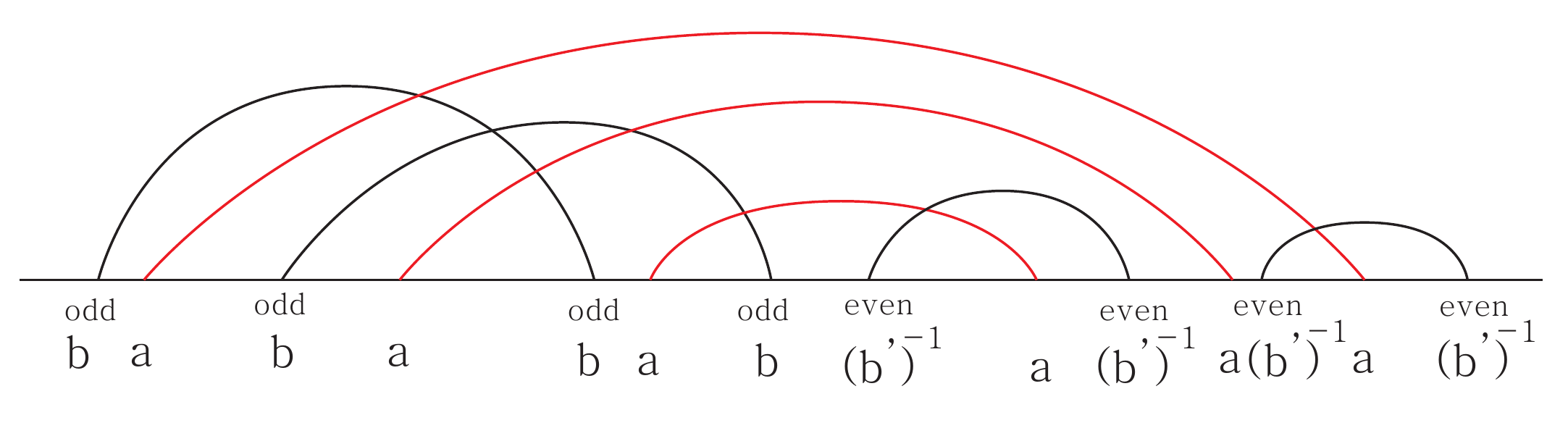}
\caption{The main example and the corresponding word in the group $G$}
\label{informative}
\end{figure}

Now, we define $l$ to be twice \footnote{We take twice the number of chords in
order to agree with \cite{Cobordisms}} the number of odd chords of the first sort minus twice the number of odd chords
of the second sort.

\begin{thm}
If $K_{1}$, $K_{2}$ represent the same long free knot then $l(K_{1})=l(K_{2})$. If two long free knots are obtained
from some free knot $K$ by breaking it at two different points then $l(K_{1})=\pm l(K_{2})$.
\end{thm}
The proof follows from the step-by-step check of the Reidemeister moves. For more details, see \cite{Cobordisms}.

Note that the here $l$ is nothing but just bare count (if we deal with compact knots, not long ones, we get just
relative signs).

\subsection{The second step}

Now, we undertake an attempt to encode the above information about the {\em bare count} of odd chords by means
of a (non-abelian) group.
Let $K$ be a {\em long} Gauss diagram. We create the word $\phi(K)$ by using the following rule:
we follow the orientation of $K$ and whenever we meet an even chord, we write $a$, when we have an
odd chord,
we write $b$ for ends of chords  of the first type or $b'$ for ends of chords of the second type
see Fig.\ref{figurebelow} below.

\begin{figure}
\centering\includegraphics[width=300pt]{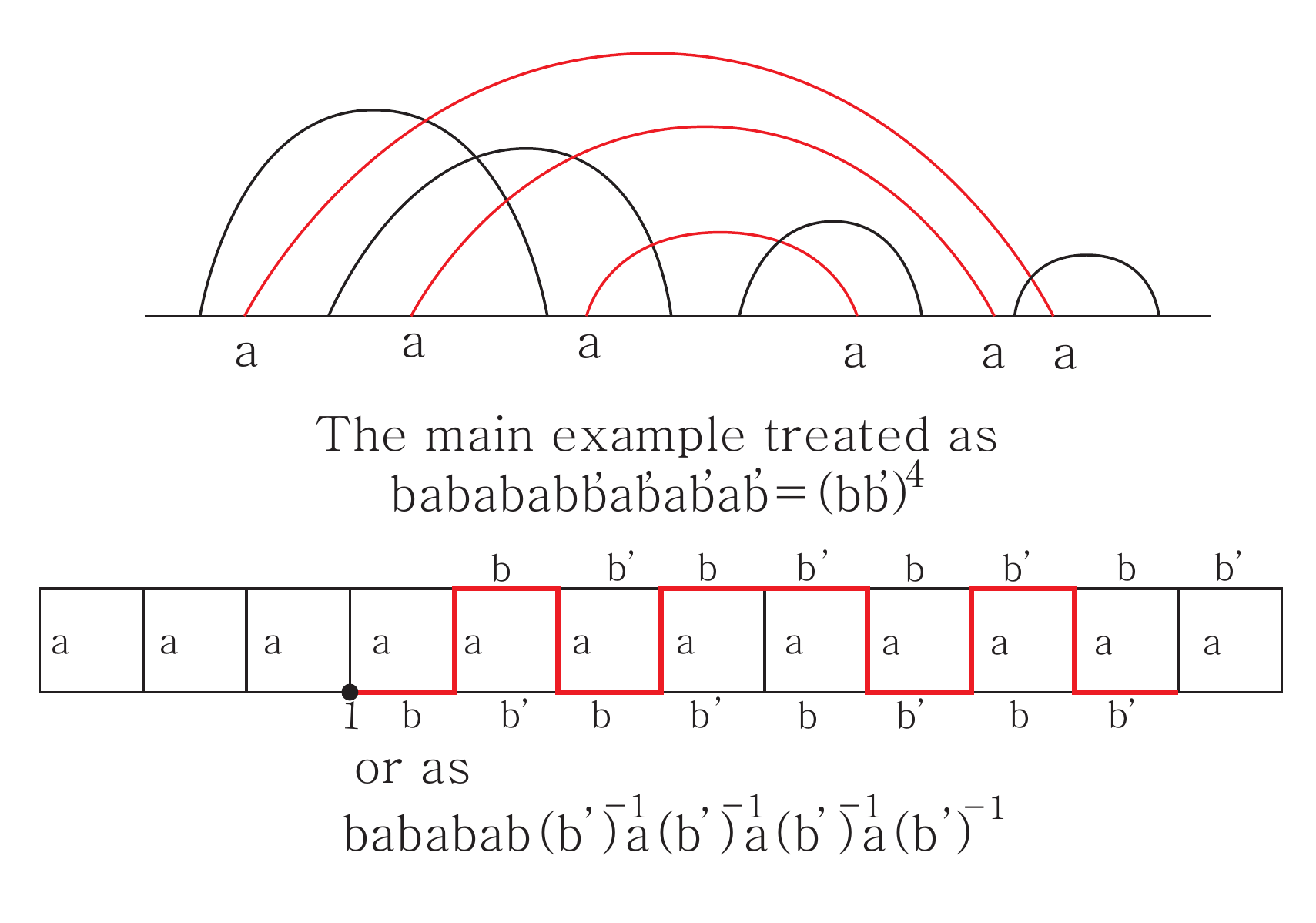}
\caption{A long free knot $K$ and the word $bababab b'ab'ab'ab'$}
\label{figurebelow}
\end{figure}

\begin{thm}\label{secondstep}
If $K_{1},K_{2}$ are equivalent as long free knots then $\phi(K_{1})=\phi(K_{2})$ in $G$.
\end{thm}

\begin{proof}
One just has to check that the first Reidemeister move gives: $aa = 1$, the second Reidemeister
move gives $aa = 1$ or $bb=1$ or $b'b'=1$; finally, the third Reidemeister move gives rise to three
substitutions each of which is either identical or $ab=b'a$ or $ab'= ba$.
\end{proof}

Hence, we see that the information from a group obtained here is not 
stronger than just the invariant $l$.

Indeed, when we look at the Cayley graph of the group $\Z$, one can see that the group has a natural bijection (which is an isomorphism on finite index subgroups) with the infinite dihedral group. However, the {\em number} $l$ now turns into a {\em geometrical}
interpretation as the {\em distance in the Cayley graph}. \footnote{Indeed, look at the bottom part of Fig. 4;
Having a word, we can consider the odd letters $b$ and $b'$ as steps to the right/ to the left; the letter a moves a point from the bottom line to the top line.
Hence, any word in $a,b,b'$ with evenly many instances of a gives rise to some element $(bb')^{k} or (b'b)^{k}$ which in turn can be considered as a horisontal shift by $\pm 2k.$}


\subsection{The third step}

Let us write more carefully at the map constructed on the second step. Can it be enhanced? In fact, we are
very much restricted in our attempt to write down all maps in terms of {\em groups} and {\em group homomorphisms}.

The letters b and b' are treated in two different ways: one way is to count them by $\pm 1$ according to the sort; the other way is to keep track of the existence of $'$ to have a word in $G$.

We can reconciliate these two approaches as follows.

Actually, we may treat $b$ in different positions differently; to describe this completely from the algebraic point of
view, one should use lots of wreath products and more elaborate machinery; we want to omit it for the first short paper.

However, one can consider the following group

$$G' =\langle a, b, b' | a^{2}=1, a b = (b')^{-1} a.\rangle $$

Note that unlike $G$, the group $G'$ has {\em exponential growth,} and $b,b'$ do not commute.

Now, the invariant
$$\psi(K)$$ is obtained from  $\phi(K)$ by using the rule:

{\em replace all letters $b'$ in even positions with $(b')^{-1}$ and replace $b$ in even
positions with $b^{-1}$.}

\begin{rk}
Actually, we can go even a bit further and get rid of the relation $a^{2}=1$ if we pass from {\em any } even chords to
{\em special} even chords, but for our purposes (exponential growth) it suffices just to work with $G'$.
\end{rk}

\begin{thm}
If two long knots $K_{1}$ and $K_{2}$ are equivalent then $\psi(K_{1})=\psi(K_{2})$ as elements of the group
$G'$.
\end{thm}

\begin{proof}
The proof follows from the direct check of the Reidemeister moves (actually, this
is a minor modification of Theorem \ref{secondstep}).
\end{proof}

Looking at Fig. \ref{figurebelow}, we readily see that for the knot $K$ given there we have

$$\psi(K) = b a b a b a b  (b')^{-1} a (b')^{-1} a (b')^{-1}a (b')^{-1} $$


\section{An example}

Let us consider the free knot $K$ given in Fig. \ref{figurebelow}.

The invariant $l=\pm 8$ says that it is not trivial and it is {\em not slice} because if has four chords
of the first sort (if we look at the compact knot, we say that we have four chords of the same sort), so
they can not be cancelled by means of Reidemeister moves and can not be spanned by a disc. On the Cayley graph
these eight chord ends sum up and give $8$; this number $8$ is invariant \footnote{There are various
cobordism obstructions which originate from the count of some ``odd things'' with writhe numbers
serving signs \cite{KauffmanParity}, however, here odd chords themselves have other signs coming from other reasonings;
the reader can readily expand these invariant for the case of virtual knots.
}.

The group $G$ actually gives the same information: though it is not abelian, it has an abelian group of {\em finite index},
hence the only information one can get in this group is just $(bb')^{k}$ (in our case $k=4$).

From the construction of the group $G'$ we immediately get the following theorem.

\begin{thm}
Let $K_{1},K_{2}$ be two long free knots. Then for the connected sum $\psi (K_{1}{\#} K_{2})= \psi(K_{1}) \cdot \lambda \psi(K_{2})$,
where $\lambda$ is some automorphism of the group $G'$.
\end{thm}

Here one can naturally see that the invariant $\psi$ can easily detect

\begin{enumerate}

\item Free knot invertibility;

\item Any sorts of free knot mutations.

\end{enumerate}

Of course, having this done for free knots, we have lots of consequences for virtual knots.
Certainly, the next goal will be to do the same for classical knots and for knots in $3$-manifolds.

\section{Further directions}

Here we sketch some directions of further investigation we are undertaking now. In each of the problems we
highlight the main problem which seems to be tractable by using the approach given in the present paper.

\begin{enumerate}

\item First, it is rather obvious that free knots {\em almost never commute}. Indeed, the group $G'$ is non-abelian;
to show that for long knots $K_{1}, K_{2}$ we have $K_{1}{\#} K_{2}\neq K_{2}{\#} K_{1}$ we have to
look carefully at the group $G'$ constructed above.

\item Certainly, the groups constructed here can be enhanced by large if we use not only parity but also parity
hierarchy; one can construct invariants of free knots valued in free products of groups whose generators are ``responsible for''
the existence of certain ``sub-knot structures''.

\item The invariant $l$ (counted with $\pm$ sign) is certainly a sliceness obstruction for free knots but
it can't say anything about the slice genus of free knots. The problem is that having a non-zero value of $l$
(say, $l=8$) and applying cut-and-paste techniques, we can ``roughly'' split $8$ into the sum $4+4$ and then
paste $4$ with $-4$ and get $0$.

Having groups which are close to free groups one can use commutator length techniques and estimate from below
the number of operations one needs to cut the word into, in order to compose a trivial element.

\item The approach given here gives a lot for virtual objects (upgrades from free knots to virtual knots
can be done straightforwardly).

We expect to have invariants for classical knots. This is to be done in a sequel of the present paper.

The key idea is to concentrate on knots inside the cylinder; this will give rise to {\em parities, indices
and other non-trivial structures}. As a result, we expect maps (not necessarily homomorphic) from classical knots
to free groups.

\end{enumerate}

\end{document}